\documentclass[11pt]{article}

\usepackage{epsfig}
\usepackage{amsfonts}
\usepackage{amssymb}
\usepackage{amstext}
\usepackage{amsmath}
\usepackage{xspace}
\usepackage{theorem}
\usepackage{graphicx}
\usepackage{tikz}
\usepackage{pgfplots}
\usepackage{enumitem}
\usepackage{setspace}
\usepackage{algorithm} 
\usepackage{algpseudocode} 
\usepackage{soul}
\usepackage{hyperref}
\usepackage{xcolor}

\makeatletter
\newenvironment{proof}{{\bf Proof:  }}{\hfill\rule{2mm}{2mm}}

\newtheorem{theorem}{Theorem}
\newtheorem{lemma}[theorem]{Lemma}
\newtheorem{proposition}[theorem]{Proposition}

\newtheorem{corollary}[theorem]{Corollary}

\newtheorem{maintheorem}{Main Theorem}

\pgfplotsset{compat=1.18}

\newcommand{\Z}{\mathbb{Z}}

\newcommand{\diff}{\mathrm{diff}}

\DeclareMathOperator{\diam}{diam}

\title{Minimum packing density for sets of four integers}
\author{
Cindy Li \thanks{Carnegie Mellon University, \texttt{cindyli@andrew.cmu.edu}} 
\and 
David Offner\thanks{Carnegie Mellon University, \texttt{doffner@andrew.cmu.edu}}
\and Iris Ye \thanks{Carnegie Mellon University, \texttt{maolinsy@andrew.cmu.edu}} 
}

\date{\today}

\begin{document}

\maketitle

\begin{abstract}
We prove that the set \( \{0, 1, 4, 6\} \) achieves the minimum packing density among all sets of integers with cardinality four, with a density of \( \frac{1}{7} \). 
\end{abstract}

\section{Introduction}
Let $\Z$ be the set of integers, and throughout the paper $S$ will always be a subset of $\Z$. Given $T \subseteq \Z$, define $S+T := \{s+t: s \in S, t \in T\}$.  Similarly, given $t \in \Z$, define $S + t = S + \{t\} = \{s+t: s \in S\}$, and $t + S := \{t\} + S = \{t+s: s \in S \}$. Note $t + S = S + t$ for all $S \subseteq \Z$, $t \in \Z$.

Given a set $S$, we say a set $A \subseteq \Z$ is \emph{$S$-packing} if for all $a_1, a_2 \in A$ s.t. $a_1 \neq a_2$, $S+a_1$ and $S+a_2$ are disjoint. For example, if $S = \{0,1,4,6\}$, then $A = \{7n: n \in \Z\}$ is an $S$-packing set.

Given a finite set $S$, we are interested in the densest set $A$ such that $A$ is $S$-packing. 
For a set $A \subseteq \Z$, define the \emph{upper density $\overline{d}(A)$} and  \emph{lower density $\underline{d}(A)$} as follows:
$$\overline{d}(A) := \limsup_{n \to \infty}{\frac{|A \cap [-n,n]|}{2n+1}} \hspace{1.5cm} \underline{d}(A) := \liminf_{n \to \infty}{\frac{|A \cap [-n,n]|}{2n+1}}.$$

If $\overline{d}(A) = \underline{d}(A)$, define the \emph{density $d(A)$} of $A$ as 
$$d(A) := \overline{d}(A) = \underline{d}(A).$$

For example, if  $A = \{7n: n\ \in \Z\}$, then $d(A) = \overline{d}(A) = \underline{d}(A) = 1/7$.

We now define the \emph{packing density} of $S$, denoted  $d_p(S)$, as 
$$d_p(S) := \sup\{\overline{d}(A): A \text{ is } S \text{-packing}\}.$$

Many researchers have studied packing alongside the dual concept of covering. Define a set \( A \subseteq \mathbb{Z} \) as \textit{$S$-covering} if $S+A = \Z$, and define the \textit{covering density} of \( S \), denoted \( d_c(S) \), as the infimum of the lower densities \( \underline{d}(A) \) of all \( S \)-covering sets \( A \subseteq \mathbb{Z} \).

Motivated by Newman \cite{d__j__newman_1967}, Schmidt and Tuller \cite{SchmidtTuller2008} proposed a conjecture regarding the packing and covering densities for integer sets of cardinality three. Later, Frankl, Kupavskii, and Sagdeev \cite{Frankl2023} proved an explicit formula for the packing and covering densities of such sets, confirming the conjecture. For integer sets with size $k$ where $k \ge 5$, Weinstein \cite{weinstein1976covering} proved the packing density is upper-bounded by $\left( \frac{k^2}{4} + \frac{11k}{6} - 14 \right)^{-1}$.

Our main result concerns the minimum packing density of any set of integers with cardinality four. 

\begin{maintheorem} 
\label{main_thm}
    If $S$ is a set of integers with cardinality four, then $d_p(S) \ge 1/7$.  Furthermore, the set $\{0,1,4,6\}$ achieves this minimum packing density: $d_p(\{0,1,4,6\}) = 1/7$.
\end{maintheorem}

The proof is given in Sections~\ref{ub} and \ref{lb}, where we prove upper bounds and lower bounds on packing density, respectively.  First, we make an observation about $S$-packing sets that will be used in both sections. Define the \emph{difference set} of $S$, denoted $\diff(S)$, as $\diff(S) := \{s-t:s,t \in S, s \ge t\}$. For example, if $S = \{0,1,3\}$, then $\diff(S) = \{0,1,2,3\}$, and if $S = \{0,2,7\}$, then $\diff(S) = \{0,2,5,7\}$.  Note that if $S$ is finite, then $|\diff(S)| \le \binom{|S|}{2} + 1$.

\begin{proposition}
\label{diff}
Suppose $S, A \subseteq \Z$.  Then $A$ is $S$-packing if and only if for all $a,b \in A$ with $a<b$, $b-a \notin \diff(S)$.
\end{proposition}

\begin{proof}
We will prove that $A$ is not $S$-packing if and only if there exists $a, b \in A$ such that $a < b$ and $b - a \in \diff(S)$. Let $a,b \in A$ be such that $a < b$ and $b-a \in \diff(S)$.  By the definition of $\diff(S)$, there exist $s,t \in S$ such that $b-a = s-t$.  Thus $a+s = b+t$.  This implies that $(b+S) \cap (a+S) \neq \emptyset$, so $A$ is not $S$-packing.

Conversely, assume $A$ is not $S$-packing.  This implies there exists $a,b \in A$ with $a<b$ and $(b+S) \cap (a+S) \neq \emptyset$, so there exist $s,t \in S$ such that $a+s = b+t$, and equivalently $b-a = s-t$.  Since $a<b$, $s>t$, so $s-t \in \diff(S)$. Hence, $b-a \in \diff(S)$.
\end{proof}

\section{An upper bound on packing density}
\label{ub}

\begin{proposition}
\label{basis}
If $\{0, \dots, n\} \subseteq \diff(S)$ for some $n \in \Z$, then $d_p(S) \le \frac{1}{n+1}$.
\end{proposition}

\begin{proof}
Suppose $A$ is $S$-packing, and $\{0, \dots, n\} \subseteq \diff(S)$ for some $n \in \Z$.  For $t \in \Z$, let $I(t)$ be the set of $n+1$ consecutive integers $\{t(n+1), t(n+1)+1, \dots, t(n+1) +n\}$. Let $a,b$ be two elements of $A$.  If $a,b \in I(t)$ with $a<b$, then $b-a \in \{0, \dots, n\} \subseteq \diff(S)$.  Thus by Proposition~\ref{diff}, at most one element of $A$ can be in $I(t)$ for each $t$.

For any $k \ge 1$, since $[-k(n+1), k(n+1)] =  \bigcup_{t=-k}^{k-1} I(t) \cup \{k(n+1)\}$,
\[\frac{|A \cap [-k(n+1),k(n+1)]|}{2k(n+1)+1} \le \frac{2k+1}{2k(n+1)+1}.\]

Thus
$$\overline{d}(A) \le \limsup_{k \to \infty}{\frac{2k+1}{2k(n+1)+1}} \le \frac{1}{n+1}\limsup_{k \to \infty}\frac{2k+1}{2k} = \frac{1}{n+1}.$$

We conclude every $S$-packing set must have upper density at most $\frac{1}{n+1}$, so by the definition of packing density, $d_p(S) \le \frac{1}{n+1}$.
\end{proof}

\begin{corollary}
\label{upperbound}
The packing density of $\{0,1,4,6\}$ is at most $\frac{1}{7}$, i.e.
    $$d_p(\{0,1,4,6\}) \le \frac{1}{7}.$$
\end{corollary}

\begin{proof}
    Since $\diff{(\{0,1,4,6\})} = \{0,1,2,3,4,5,6\}$, applying Proposition \ref{basis} with $n = 6$, we get $d_p(\{0,1,4,6\}) \le \frac{1}{6+1} = \frac{1}{7}$.
\end{proof}

\section{A lower bound on packing density}
\label{lb}

In this section, given a finite set of integers $S$, we describe a ``greedy packing algorithm'' to generate an $S$-packing set $A$, while proving properties of this set along the way. The set $A$ may not be the densest $S$-packing set, but we will show that it has density at least $1/|\diff(S)|$. For every $i \ge 0$, the algorithm iteratively produces sequences $\{t_i\}$ and $\{A_i\}$, where $t_i \in \Z$ and $A_i$ is $S$-packing, as follows:  Initially, set $A_{-1} = \emptyset$. Then, for $i \ge 0$, define $A_i = A_{i-1} \cup \{t_i\}$, where
$$t_i = \min \{x :x \ge 0, x \notin (A_{i-1} + \diff(S))\}.$$

Note $t_0=0$, and for all $i \ge -1$, $A_i  \subseteq A_{i+1}$.  Thus,  $A_i + \diff(S) \subseteq A_{i+1}+\diff(S)$, so since in each step $t_i \ge 0$ is chosen to be minimal, $t_{i} < t_{i+1}$.

As an example, if $S = \{0,4,5\}$, then $t_0 = 0$, $t_1 = 2$, $t_2 = 8$, $t_3 = 10$, and so forth.  In this example $A_0= \{0\}$, $A_1= \{0,2\}$, $A_2= \{0,2,8\}$, $A_3= \{0,2,8,10\}$, and so forth.

\begin{lemma}
\label{ti_UB}
    Let $S$ be a finite set. For all $i \ge 0$, $t_i \le i|\diff(S)|$.
\end{lemma}
\begin{proof}
For $i \ge 0$, since $t_i$ is the smallest nonnegative integer not in $A_{i-1} + \diff(S)$, $t_i \le |A_{i-1}+\diff(S)|$.
Thus it suffices to show that for all $i \ge 0$, 
    $$|A_{i-1}+\diff(S)| \le i|\diff(S)|.$$ 
This follows directly from the fact that for $i\ge 0$, $A_{i-1} = \{t_0, t_1, \dots, t_{i-1}\}$ has cardinality $i$. Thus

\[|A_{i-1}+\diff(S)| = |\{t_0, t_1, \dots, t_{i-1}\}+\diff(S)|
= |\bigcup_{k=0}^{i-1} t_{k}+\diff(S)| \le \sum_{k=0}^{i-1} |t_{k}+\diff(S)| = i|\diff(S)|.\]
\end{proof}

Next, define the set $A_\infty$ to be the union of all $A_i$'s produced by the greedy packing algorithm:
$$A_{\infty} = \bigcup_{i \ge 0} A_i.$$

In the example where $S = \{0,4,5\}$, $A_\infty = \{0, 2, 8, 10, 16, 18, 24, 26, \dots \}$.

Define a set $A \subseteq \Z$ to be \emph{periodic} if there is some positive integer $p$ such that for all $i \in \Z$, 
\[i \in A \iff i+p \in A,\] 
and we say $A$ has period $p$. Given $A \subseteq \Z$ and integers $a$ and $b$ with $b>a$, define the \emph{interval density $d(A, a, b)$} of $A \subseteq \Z$ restricted to the interval $[a,b-1] \subseteq \Z$ as
$$d(A, a, b) := \frac{|A \cap [a,b-1]|}{b-a}.$$

Note that if $A$ is periodic with period $p$, $d(A)$ is defined and $d(A) = d(A,a,a+p)$ for any $a \in \Z$. For $a \in \Z$ and $p \in \Z$ with $p>0$, define a set $A \subseteq \Z$ to be \emph{$a$-forward-periodic with period $p$} if, for all $i \ge 0$, $$a+i \in A \iff a+i+p \in A.$$
For example, the set $\{0, 2, 8, 10, 16, 18, 24, 26, \dots \}$ is $0$-forward periodic with period $8$.

 For a set of integers $S$, define the \emph{diameter} of $S$, denoted $\diam(S)$, as 
 $$\diam(S) := \sup{\{s-t: s,t \in S\}}.$$

\begin{lemma}
\label{Ainf_FP}
    There exist $a, b \ge 0$ with $a < b$ such that $A_\infty$ is $a$-forward periodic with period $b-a$.
\end{lemma}

\begin{proof}
    Let's first construct $a$ and $b$. For $t,i \in \Z$, define $f : \Z \times \Z \to \{0,1\}$ so that
    \begin{equation*}
        f(t,i)=
        \begin{cases}
            1 & \text{if } t+i \in A_\infty \\
            0 & \text{otherwise}
        \end{cases}
    \end{equation*}
Further, define $F: \Z \to \{0,1\}^{\diam(S)}$ to be the function where
    $$F(t) = \left(f(t,0), \dots, f(t,\diam(S)-1)\right)$$
    for all $t \in \Z$. Since the codomain $\{0,1\}^{\diam(S)}$ of $F$ is finite, by the pigeonhole principle, there must be $a,b$ where $0 \le a < b$ such that $F(a) = F(b)$.  For these values of $a$ and $b$, we prove that $A_\infty$ is $a$-forward periodic with period $b-a$ by induction on $i$. That is, we will prove that for all $i \ge 0$,
    \[a+i \in A_\infty \iff b+i \in A_\infty.\]
    
    Base cases ($i=0, 1, \dots, \diam(S)-1$): Recall $a$ and $b$ are fixed such that $F(a) = F(b)$. By definition of $F$, it means for all $0 \le i < \diam(S)$, $a + i \in A_{\infty}$ if and only if $b + i \in A_{\infty}$. 

    Inductive step: For the induction hypothesis, let \( m \geq \diam(S) - 1 \), and assume for all \( 0 \leq i \leq m \) that
    \[
    a + i \in A_\infty \iff b + i \in A_\infty.
    \]
    We now show $a+(m+1) \in A_\infty \iff b+(m+1) \in A_\infty$. 
    
    Suppose $a+(m+1)\notin A_\infty$. Since $a+(m+1)$ is not chosen in the algorithm, $a+(m+1) \in A_\infty + \diff(S)$. Thus, there exists $a_0 \in A_\infty$ and $s, t \in S$ with $s \ge t$ such that $a_0 + s-t = a+(m+1)$. Since $a_0 \in A_\infty$ and $a+(m+1) \notin A_\infty$, $s \neq t$, so $s>t$.

    Note $a_0 = a + (m+1) -(s-t)$. Since $s-t \in \diff(S)$ and $s \neq t$, $1 \le s-t \le \diam(S)$. Also, note $m+1 \ge \diam(S)$. Thus, $0 \le (m+1)-(s-t) \le m$.

    Let $b_0 = b+(m+1)-(s-t)$. By the induction hypothesis, $b_0 \in A_\infty$. Thus, $b+ (m+1) \in A_\infty + \diff(S)$, and by the construction of $A_\infty$, $b+(m+1)$ would not be chosen by the greedy packing algorithm, and $b+(m+1) \notin A_\infty$.

    The same argument shows that if $b+(m+1) \notin A_\infty$ then $a+m+1 \notin A_\infty$. Thus $A_\infty$ is $a$-forward periodic with period $b-a$.
\end{proof}

For the remainder of the section, let  $a$ and $b$ be as in Lemma \ref{Ainf_FP}, i.e. $A_\infty \subseteq \Z$ is $a$-forward periodic with period $b-a$.

\begin{lemma}
\label{Ainf_FP2}
If $a<b \in \Z$, and $A_\infty \subseteq \Z$ is $a$-forward periodic with period $b-a$, then  
    $$\lim_{n \to \infty}d(A_\infty, 0, n) = d(A_\infty, a, b).$$
\end{lemma}

\begin{proof}
    \begin{align*}
        \lim_{n \to \infty}d(A_\infty, 0, n) =& \lim_{n \to \infty}{\frac{|A_\infty \cap [0, n-1]|}{n}} \\
        =& \lim_{n \to \infty}{\frac{|A_\infty \cap [0, a-1]|}{n} + \lim_{n \to \infty}{\frac{|A_\infty \cap [a, n-1]|}{n}}} \\
        =& \lim_{n \to \infty}{\frac{|A_\infty \cap [a, n-1]|}{n}} \\
        =& \lim_{n \to \infty}\left({\frac{\lfloor{\frac{n-a}{b-a}\rfloor}|A_\infty \cap [a, b-1]|}{n} + o(1)} \right)\\
        =& \lim_{n \to \infty}{\left(\frac{n-a}{n}\right) \left(\frac{|A_\infty \cap [a, b-1]|}{b-a}\right)} \\
        =& \frac{|A_\infty \cap [a, b-1]|}{b-a} \\
        =& d(A_\infty, a, b)
    \end{align*}
\end{proof}

Given $x \in \Z$, and $a < b \in \Z$, define the \emph{remainder} $r(x)$ to be the integer where $0 \le \emph{r(x)} < b-a$ and 
$$x-a \equiv r(x) \mod(b-a).$$

We are finally ready to define the $S$-packing set $A$ produced by the greedy packing algorithm.  Given $A_\infty$, $a$, and $b$ as in Lemma \ref{Ainf_FP}, define $A \subseteq \Z$ to be the set such that 
\[x \in A \iff a + r(x) \in A_\infty \quad \text{for all} \quad x \in \mathbb{Z}.
\]

For example, if $S = \{0,4,5\}$, $A = \{\dots, -16, -14, -8, -6, 0, 2, 8, 10, 16, 18, 24, 26, \dots \}$.

By construction, $A$ is periodic with period $b-a$, and therefore  $d(A) = d(A, a, b)$. Furthermore, since for $i \ge 0$, $a+i \in A \iff a+i \in A_\infty$, $d(A,a,b) = d(A_\infty, a, b)$. Combining the two equations, we have
$$d(A) = d(A_\infty, a, b).$$

\begin{proposition}
    \label{d(A)_LB} Let $A$ be the set produced by the greedy packing algorithm.  Then $A$ is $S$-packing with $d(A) \ge \frac{1}{|\diff(S)|}$.
\end{proposition}

\begin{proof}
    We first prove that $A$ is $S$-packing. Assume for the sake of contradiction that $A$ is not $S$-packing. Then, by Proposition ~\ref{diff}, there exist $c$, $d \in A$ with $c < d$ and $d-c \in \diff(S)$. By the construction of $A$, since $c$ and $d$ are in $A$, we know $a+r(c)$ and $a+r(d)$ are in $A_\infty$. Since $d-c \equiv r(d)-r(c) \pmod{b-a}$, we know $(d-c) + r(c) \equiv r(d) \pmod{b-a}$.  Thus since $A_\infty$ is $a$-forward periodic with period $b-a$, and $a + r(d) \in A_\infty$, we know $a + r(c) + (d-c) \in A_\infty$.  Since $a+r(c)$ and $a+r(c) + (d-c)$ are both in $A_\infty$, with  $d-c \in \diff(S)$, this contradicts the fact that, by construction, $A_\infty$ is $S$-packing. Thus $A$ is $S$-packing.
    
    Then, we show $d(A) \ge \frac{1}{|\diff(S)|}$. For all $i \ge 0$, $t_i \le i|\diff(S)|$ by Lemma \ref{ti_UB}. Hence, for all $n \ge 1$,
    $$|A_\infty \cap [0,n-1]| \ge \frac{n}{|\diff(S)|}.$$ 
    Rearranging the inequality, we have that for all $n \ge 1$,
    \[d(A_\infty, 0, n) = \frac{|A_\infty \cap [0,n-1]|}{n} \ge \frac{1}{|\diff(S)|}.\]

    Hence, 
    $$\lim_{n \to \infty}d(A_\infty, 0, n) \ge \lim_{n \to \infty}\frac{1}{|\diff(S)|} = \frac{1}{|\diff(S)|}.$$

    By Lemma \ref{Ainf_FP2}, 
    $$d(A_\infty, a, b) \ge \frac{1}{|\diff(S)|}.$$

    Since we have already shown that $d(A) = d(A_\infty,a,b)$, we conclude
    $$d(A) \ge \frac{1}{|\diff(S)|}.$$

    Thus $A$ is $S$-packing and $d(A) \ge \frac{1}{|\diff(S)|}$.
\end{proof}

\begin{corollary}
\label{gp_thm}
Every finite set $S \subseteq \Z$ has $d_p(S) \ge \frac{1}{|\diff(S)|}$.
\end{corollary}
\begin{proof}
    This follows directly from Proposition \ref{d(A)_LB}.
\end{proof}

\begin{corollary}
\label{lowerbound}
The packing density of $\{0,1,4,6\}$ is lower bounded by $\frac{1}{7}$, i.e.
    $$d_p(\{0,1,4,6\}) \ge \frac{1}{7}.$$
\end{corollary}

\begin{proof}
    Since $\diff{(\{0,1,4,6\})} = \{0,1,2,3,4,5,6\}$, applying Corollary \ref{gp_thm},
    $$d_p(\{0,1,4,6\}) \ge \frac{1}{|\diff{(\{0,1,4,6\})}|} = \frac{1}{7}.$$ 
\end{proof}

Note that for $S = \{0,1,4,6\}$, the greedy packing algorithm produces the $S$-packing set $A = \{7n: n \in \Z\}$, which is densest possible by Corollary \ref{upperbound}.

We now prove the main theorem.
\begin{maintheorem} 
    If $S$ is a set of integers with cardinality four, then $d_p(S) \ge \frac{1}{7}$.  Furthermore, the set $\{0,1,4,6\}$ achieves this minimum packing density: $d_p(\{0,1,4,6\}) = \frac 1 7$.
\end{maintheorem}
\begin{proof}
    Let $S \subset \Z$ be an arbitrary set with $4$ integers. By Corollary~\ref{gp_thm}, $d_p(S) \ge \frac{1}{|\diff(S)|}$. Since $|\diff(S)| \le 7$ for all integer sets with cardinality $4$, we conclude
        $$d_p(S) \ge \frac{1}{7}.$$
    At the same time, combining the inequalities in Corollary \ref{upperbound} and \ref{lowerbound}, we have 
        $$d_p(\{0,1,4,6\}) = \frac{1}{7}.$$
    Hence, $\{0,1,4,6\}$ achieves the minimum packing density among all sets of four integers.
\end{proof}

\end{document}